\RequirePackage{fix-cm}
\documentclass[smallextended]{svjour3}       
\smartqed  
\usepackage{graphicx}

\usepackage{amsfonts}
\usepackage{amssymb}
\usepackage{amsmath}
\usepackage{mathrsfs}

\newcommand{\R}{\mathbb R}
\newcommand{\N}{\mathbb N}

\newcommand{\X}{\mathbb X}
\newcommand{\Y}{\mathbb Y}
\newcommand{\Uball}{{\mathbb B}}
\newcommand{\Usfer}{{\mathbb S}}

\newcommand{\dom}{{\rm dom}\, }
\newcommand{\gr}{{\rm graph}\,}

\newcommand{\nullv}{\mathbf{0}}

\newcommand{\inte}{{\rm int}\, }

\newcommand{\alphao}{\alpha_0}

\newcommand{\Lin}{{\mathcal L}}
\newcommand{\Hom}{{\mathcal H}}
\newcommand{\Kern}{{\rm Ker}\,}
\newcommand{\Fan}{{\mathcal A}}
\newcommand{\GEq}{{\tt (GE)}}
\newcommand{\AGEq}{{\tt (AGE)}}

\newcommand{\ball}[2]{{\rm B}(#1, #2)}

\newcommand{\cldis}[1]{|#1|^\downarrow}
\newcommand{\grsl}[1]{#1^\downarrow}
\newcommand{\FDer}[1]{{\rm D}#1}
\newcommand{\Fsubd}{\widehat{\partial}}

\newcommand{\dist}[2]{{\rm dist}\left(#1,#2\right)}
\newcommand{\subregmap}[2]{{\rm subreg}\, #1(#2)}
\newcommand{\calmap}[2]{{\rm clm}\, #1(#2) }
\newcommand{\intrad}[2]{\varrho(#1)(#2)}
\newcommand{\Solv}[2]{S_{#1,#2}}

\newcommand{\subreg}[3]{{\rm subreg}\, #1(#2,#3)}
\newcommand{\calm}[3]{{\rm clm}\, #1(#2,#3) }

\begin{document}

\title{A strong metric subregularity analysis of nonsmooth mappings
via steepest displacement rate}

\titlerunning{Strong metric subregularity of nonsmooth mappings}       

\author{A. Uderzo}


\institute{A. Uderzo \at
                   Dept. of Mathematics and Applications \\
                   University of Milano-Bicocca \\
              Tel.: +39-02-64485871\\
              Fax: +39-02-64485705 \\
              \email{\tt amos.uderzo@unimib.it} 
}

\date{\today}

\maketitle

\begin{abstract} 
In this paper, a systematic study of the strong metric subregularity
property of mappings is carried out by means of a variational
tool, called steepest displacement rate. With the aid of this tool, a simple
characterization of strong metric subregularity for multifunctions
acting in metric spaces is formulated. The resulting criterion
is shown to be useful for establishing stability properties of the
strong metric subregularity in the presence of perturbations, as well as for
deriving various conditions, enabling to detect such a property
in the case of nonsmooth mappings.
Some of these conditions, involving several nonsmooth analysis
constructions, are then applied in studying the isolated calmness
property of the solution mapping to parameterized generalized
equations.
\end{abstract}



\section{Introduction}
\label{intro}

Several remarkable advances in optimization have been made
possible in recent years
thanks to a deepened understanding of stability properties of
multifunctions. In fact, their study has gained a well-recognized place within
modern variational analysis. Among the properties of multifunctions
mainly applied in optimization and related topics, those describing
a Lipschitzian behaviour play a crucial role. Under this category falls
metric regularity, which is most likely the best known and widely
employed. Nonetheless, it turns out that metric regularity is not
strictly requested in certain circumstances, while its work can
be done by a weaker property called metric subregularity, of
course at a lower price in terms of problem assumptions.
Consider, for instance, the algebraic characterization of the
tangent space to a manifold, which is defined by an equation expressed by
a smooth mapping. According to a standard argument, 
this is the key tool for deriving the
Euler-Lagrange multiplier rule in nonlinear optimization,
often presented among the consequences of the celebrated
Lyusternik's theorem (see \cite{AlTiFo87,BorZhu05,IofTik79}).
In order to establish the non-trivial inclusion (the kernel of the
derivative is contained in the tangent space) metric regularity
is usually invoked, even though the mere metric subregularity
would be enough. As another example, consider the exact
penalization principle for constrained optimization problem
with Lipschitz objective function (see \cite{FacPan03}). It happens that
this principle can be invoked, provided  that a certain error
bound inequality is valid, and for the latter circumstance
the metric subregularity of the constraining mapping is enough.
All of this contributed to raise a large interest in metric subregularity, on
which a dedicated vast literature does now exist
(see \cite{DonRoc09,HenOut05,Krug15,NgaThe01,NgaTin15,ZheNg10}
and references therein).

The main drawback of metric subregularity is its lack of robustness
under (even small) perturbations. More precisely, it  has been
observed that such property happens to be broken if adding
to a metrically subregular mapping (even single-valued and smooth)
a Lipschitz term, yet with a small Lipschitz constant. This 
well-known phenomenon
explains the difficulty in employing perturbation schemes,
when studying criteria for detecting metric subregularity.
In this paper, a systematic study is proposed of a special variant
of metric subregularity, called strong metric subregularity,
which is known to exhibit a notable robustness quality, while
keeping rather low requirements in comparison with metric
regularity. In particular, the present study concentrates on
sufficient conditions
for the strong metric subregularity of (possibly) nonsmooth
mappings. The analysis of this topic is performed by making use
of a variational tool called steepest displacement rate, that
enables to formulate a general criterion already in a metric space setting.
Such an approach leads to a unifying scheme of analysis
and emphasizes the connection of the property under study with
the notion of local sharp minimality.
Strong metric subregularity along with its stability properties
and related infinitesimal characterizations have been considered
recently by several authors. For an account on various aspects
of the emerging theory of strong metric subregularity, the
reader may refer to \cite{DonRoc09}.

The contents of the paper are exposed according to the following
structure. In Section \ref{sec:2}, after the notion of strong metric
subregularity is presented and its reformulation in terms of
isolated calmness property for the inverse mapping is recalled, the notion
of steepest displacement rate is introduced and exploited to establish
the basic characterization. The connection with local sharp minimality
is also discussed, while several situations arising in different topics
are illustrated, aimed at providing motivations for the interest in the
main subject of the paper. In Section \ref{sec:3}, two relevant
manifestations of the robustness behaviour of strong metric
subregularity are embedded in the framework of the steepest
displacement rate analysis. In Section \ref{sec:4} several known tools
of nonsmooth analysis are combined with the main criterion in order
to obtain conditions for the strong metric subregularity of nonsmooth
mappings. Some of these results are then applied in Section \ref{sec:5}
to investigate the isolated calmness property of the solution mapping
to parameterized generalized equations, with base and field term.
A final section is reserved for general comments on the exposed
achievements.


\section{Strong metric subregularity and its equivalent reformulations}
\label{sec:2}

Let us start with recalling the main properties under study. This
will be done in a metric space setting, which is the natural environment
where the Lipschitzian analysis of stability of multifunctions can
be conducted. In a metric space $(X,d)$, the distance from a
point $x\in X$ to a subset $S\subseteq X$ is denoted by $\dist{x}{S}$,
with the convention that $\dist{x}{\varnothing}=+\infty$,
while $\ball{x}{r}$ denotes the closed ball with center $x$ and
radius $r$.

\begin{definition}      \label{def:strmsubreg}
(i) A set-valued mapping $F:X\rightrightarrows Y$ between metric spaces is called
{\it metrically subregular} at $(\bar x,\bar y)\in\gr F$ if there exist $\kappa\ge 0$ and
$r>0$ such that
\begin{eqnarray}  \label{in:msubreg}
      \dist{x}{F^{-1}(\bar y)}\le\kappa\,\dist{\bar y}{F(x)},\quad
     \forall x\in\ball{\bar x}{r}.   
\end{eqnarray}
Denote by
$$
   \subreg{F}{\bar x}{\bar y}=\inf\{\kappa\ge 0:\ \exists r>0
   \hbox{ satisfying } (\ref{in:msubreg})\}
$$
the {\it modulus of subregularity} of $F$ at $(\bar x,\bar y)$.
Whenever $F$ is single-valued, the simpler notation
$\subregmap{F}{\bar x}$ will be used.

(ii) A set-valued mapping $F:X\rightrightarrows Y$ between metric spaces is called
{\it strongly metrically subregular} at $(\bar x,\bar y)\in\gr F$ if $F$ is metrically
subregular at $(\bar x,\bar y)$ and, in addition, $\bar x$ is an isolated point
of $F^{-1}(\bar y)$ or, equivalently, if there exist $\kappa\ge 0$ and $r>0$
such that
\begin{eqnarray}  \label{in:strmsubreg}
     d(x,\bar x)\le\kappa\,\dist{\bar y}{F(x)},\quad\forall x\in\ball{\bar x}{r}.   
\end{eqnarray}
\end{definition}

Roughly speaking, whereas the well-known metric regularity property of
a mapping $F$ at $(\bar x,\bar y)\in\gr F$ can be
viewed as a quantitative form of local solvability for the inclusion
$y\in F(x)$, the strong metric subregularity corresponds to a
quantitative form of local uniqueness for the solution $\bar x$
to the particular inclusion $\bar y\in F(x)$. The independence
of these two properties is illustrated in the next example.

\begin{example}
Let $X=Y=\R$ be equipped with its usual Euclidean metric structure.
Consider the mapping $F_1:\R\rightrightarrows\R$ defined by
\begin{eqnarray*}
   F_1(x)=\left\{\begin{array}{ll}
                  [0,1/2), & \hbox{ if } x=0, \\                                                              
                  {}[1,+\infty),  &  \hbox{ otherwise}.
                  \end{array} \right.
\end{eqnarray*}
Clearly, $F_1$ is strongly metrically subregular at $(0,0)$, with
$\subreg{F_1}{0}{0}=0$, but it fails to be metrically regular
near the same point. Notice that $F_1$ has not closed graph
and it is not semicontinuous.

In the same setting, consider the mapping $F_2:\R\rightrightarrows\R$
defined by $F_2(x)=[x,+\infty)$. This multifunction is metrically
regular near $(0,0)$, but it is not strongly metrically subregular
at the same point.
\end{example}

The basic tool of analysis in use throughout the present section
is introduced in the next definition.

\begin{definition}     \label{def:stdisrate}
(i) Given a function $\varphi:X\longrightarrow\R\cup\{\pm\infty\}$
defined on a metric space and an element $\bar x\in\dom f$,
the value (possibly infinite)
$$
    \grsl{\varphi}(\bar x)=\liminf_{x\to\bar x}{\varphi(x)-\varphi(\bar x)
    \over d(x,\bar x)}
$$
is called the {\it steepest descent rate} of $\varphi$ at $\bar x$.

(ii) Let $F:X\rightrightarrows Y$ be a set-valued mapping between
metric spaces and let $(\bar x,\bar y)\in\gr F$. The (nonnegative,
possibly infinite) quantity
$$
    \cldis{F}(\bar x,\bar y)=\grsl{\dist{\bar y}{F(\cdot)}}(\bar x)
$$
is called the {\it steepest displacement rate} of $F$ at $(\bar x,\bar y)$.
\end{definition}

\begin{remark}
The employment of the steepest descent rate in connection with
variational problems is witnessed since \cite{Mari82}, whereas
its application to nondifferentiable optimization goes back at least
to \cite{Gian89}. It was with V.F. Demyanov that it became steady
exploited for formulating optimality conditions in metric space
settings, as a starting point for more involved nonsmooth analysis
constructions (see \cite{Demy00,Demy05,Demy05b,Demy10,Zasl13,Zasl14}).
The use of the distance function from images of a given multifunction
to characterize its Lipschitzian properties follows the spirit
of \cite{Rock85}.
\end{remark}

A first basic characterization of strong metric subregularity is
established next as a positivity condition on the steepest
displacement rate of a given multifunction.

\begin{proposition}      \label{pro:strmsubregchar}
A set-valued mapping  $F:X\rightrightarrows Y$ is strongly metrically
subregular at $(\bar x,\bar y)\in\gr F$ if and only if
\begin{eqnarray}   \label{in:posstdisrate}
    \cldis{F}(\bar x,\bar y)>0.
\end{eqnarray}
Moreover, whenever inequality (\ref{in:posstdisrate}) holds true,
it results in
\begin{eqnarray}      \label{eq:subregstdisrate}
     \subreg{F}{\bar x}{\bar y}={1\over \cldis{F}(\bar x,\bar y)},
\end{eqnarray}
with the convention that $1/+\infty=0$.
\end{proposition}

\begin{proof}
Necessity: according to Definition \ref{def:strmsubreg}(ii), corresponding
to an arbitrary $\kappa>\subreg{F}{\bar x}{\bar y}$ there exists $r>0$
such that
$$
    {\dist{\bar y}{F(x)}\over d(x,\bar x)}\ge{1\over\kappa},
    \quad\forall x\in\ball{\bar x}{r}\backslash\{\bar x\},
$$
whence
$$
    \cldis{F}(\bar x,\bar y)\ge{1\over\kappa}.
$$ 
This evindently implies inequality (\ref{in:posstdisrate})
and, by arbitrariness of $\kappa$, the inequality
\begin{eqnarray}    \label{in:subregge1over}
    \subreg{F}{\bar x}{\bar y}\ge{1\over \cldis{F}(\bar x,\bar y)}.
\end{eqnarray}

Sufficiency: according to Definition \ref{def:stdisrate}(ii),
in the case $\cldis{F}(\bar x,\bar y)=+\infty$,
for every $\eta>0$ there exists $r_\eta>0$ such that
$$
   {1\over\eta}\dist{\bar y}{F(x)}\ge d(x,\bar x),\quad\forall
    x\in\ball{\bar x}{r_\eta}.
$$
Thus $F$ is strongly metrically subregular at $(\bar x,\bar y)$
with $\subreg{F}{\bar x}{\bar y}\le 1/\eta$, what leads to
$\subreg{F}{\bar x}{\bar y}=0$.
In the case $\cldis{F}(\bar x,\bar y)<+\infty$, for an arbitrary
$\epsilon\in (0,\cldis{F}(\bar x,\bar y))$, there is $r_\epsilon>0$
such that
$$
   {1\over\cldis{F}(\bar x,\bar y)-\epsilon}\dist{\bar y}{F(x)}
  \ge d(x,\bar x),\quad\forall  x\in\ball{\bar x}{r_\epsilon},
$$
which shows that $F$ is strongly metrically subregular at
$(\bar x,\bar y)$ and, by arbitrariness of $\epsilon$, leads to
\begin{eqnarray}    \label{in:subregle1over}
  \subreg{F}{\bar x}{\bar y}\le{1\over \cldis{F}(\bar x,\bar y)}.
\end{eqnarray}
Then, as inequalities (\ref{in:subregge1over}) and
(\ref{in:subregle1over}) are both valid now, one obtains
(\ref{eq:subregstdisrate}), thereby completing the proof.
\hfill$\square$
\end{proof}

\begin{remark}     \label{rem:smsubshamin}
Let us recall that after \cite{Poly87} an element $\bar x\in\dom\varphi$ is said to
be a {\it local sharp minimizer} of a function  $\varphi:X\longrightarrow
\R\cup\{\pm\infty\}$ defined on a metric space if there exist positive
$\zeta$ and $r$ such that
$$
    \varphi(x)\ge\varphi(\bar x)+\zeta d(x,\bar x),\quad\forall
    x\in\ball{\bar x}{r}.
$$
Clearly, $\bar x$ is a local sharp minimizer of $\varphi$ if and
only if $\grsl{\varphi}(\bar x)>0$. Thus, on account of Definition
\ref{def:stdisrate} and of Proposition \ref{pro:strmsubregchar},
a set-valued mapping $F$ is strongly metrically subregular at
$(\bar x,\bar y)$ if and only if $\bar x$ is a local sharp minimizer
of the displacement function $x\mapsto\dist{\bar y}{F(x)}$.
Notice that the positivity of the steepest descent rate of a
function is a circumstance essentially connected, in more
structured settings, with nonsmoothness (see \cite{Uder15}).
\end{remark}

Another characterization of the main property under study
for a multifunction $F:X\rightrightarrows Y$ can be obtained
through the following stability behaviour of its inverse $F^{-1}:Y
\rightrightarrows X$, i.e. $F^{-1}(y)=\{x\in X: y\in\ F(x)\}$.

\begin{definition}      \label{def:calmness}
(i) A set-valued mapping $G:X\rightrightarrows Y$ between metric spaces is called
{\it calm} at $(\bar x,\bar y)\in\gr G$ if there exists $\kappa\ge 0$ and
$r>0$ such that
\begin{eqnarray}    \label{in:calm}
      \sup_{y\in G(x)\cap\ball{\bar y}{r}}\dist{y}{G(\bar x)}\le\kappa\, d(x,\bar x),
     \quad\forall x\in\ball{\bar x}{r}.
\end{eqnarray}
Denote by
$$
   \calm{G}{\bar x}{\bar y}=\inf\{\kappa\ge 0:\  \exists r>0
   \hbox{ satisfying } (\ref{in:calm})\}
$$
the {\it calmness modulus} of $G$ at $(\bar x,\bar y)$
($\calmap{G}{\bar x}$ whenever $G$ is a single-valued mapping).

(ii) A set-valued mapping $G:X\rightrightarrows Y$ between metric spaces is
said to have the {\it isolated calmness property} at $(\bar x,\bar y)\in\gr G$
if $G$ is calm at $(\bar x,\bar y)$ and, in addition, $\bar y$ is an isolated
point of $G(\bar x)$.

(iii) A function $\varphi:X\longrightarrow\R\cup\{\pm\infty\}$ is called
{\it calm from below} at $\bar x\in\dom f$ if
$$
   \grsl{\varphi}(\bar x)>-\infty.
$$
\end{definition}

Isolated calmness seems to have made its first formal
appearance in \cite{Dont95}, where it was called ``upper-Lipschitz
property at a point" (see also \cite{DonRoc09}).

\begin{theorem}[\cite{DonRoc09}]
A set-valued mapping  $F:X\rightrightarrows Y$ is strongly metrically
subregular at $(\bar x,\bar y)\in\gr F$ if and only if $F^{-1}$ has the
isolated calmness property at $(\bar y,\bar x)$. In this case, it holds
$$
   \calm{F^{-1}}{\bar y}{\bar x}=\subreg{F}{\bar x}{\bar y}.
$$
\end{theorem}

Below several situations, connecting different topics of optimization
and variational analysis, are illustrated, where the strong metric
subregularity naturally emerges. A further relevant motivation for
being interested in strong metric subregularity has to do with
the analysis of the solution mapping to generalized equations.
This topic will be discussed in Section \ref{sec:5}.

\begin{example}
From Remark \ref{rem:smsubshamin} it should be clear that
every scalar function $\varphi:X\longrightarrow\R\cup\{\pm\infty\}$
defined in a metric space is strongly metrically subregular
at each of its local sharp minimizers (if any). As an obvious
consequence, the related epigraphical set-valued mapping
$F_\varphi:X\rightrightarrows\R$, defined as
$$
   F_\varphi(x)=[\varphi(x),+\infty),
$$
is strongly metrically subregular at $(\bar x,\varphi(\bar x))$,
whenever $\bar x$ is a local sharp minimizer of $\varphi$.
It is worth noting that if $\varphi$ is calm from below at $\bar x$,
then it can be perturbed in such a way to have the point $\bar x$ as
a local sharp minimizer. Indeed, if for some $l>0$ it is
$\grsl{\varphi}(\bar x)>-l$, then function $\varphi+ ld(\cdot,\bar x)$
admits a local sharp minimizer at $\bar x$. Now, for a lower
semicontinuous (henceforth, for short, l.s.c.) proper function
$\varphi:X\longrightarrow\R\cup\{+\infty\}$ defined in a complete
metric space, the set of all points at which $\varphi$ is calm from
below is large enough. In fact, as a direct consequence of the
Ekeland variational principle, it is possible to prove that such
set is dense in $\dom\varphi$. All of this should show that it is
not difficult to generate situations where strong metric subregularity
appears.
\end{example}

\begin{example}
Let $\Lin(\X,\Y)$ denote the space of all linear bounded operators
between two normed spaces, having null vector $\nullv$. Given
$\Lambda\in\Lin(\X,\Y)$, its {\it injectivity constant} is defined as
$$
    \alpha(\Lambda)=\inf_{\|u\|=1}\|\Lambda u\|
$$
(see, for instance, \cite{Peno13}). By linearity, for any pair
$(\bar x,\bar y)\in\X\times\Y$, with $\bar y=\Lambda\bar x$,
one finds
\begin{eqnarray}    \label{eq:injconlinop}
   \cldis{\Lambda}(\bar x,\bar y)=\cldis{\Lambda}(\nullv,\nullv)=
   \liminf_{x\to \nullv}{\|\Lambda x\|\over\|x\|}=\alpha(\Lambda).
\end{eqnarray}
While any bounded linear operator $\Lambda$ is known to be
metric subregular at each point of its graph, according to
Proposition \ref{pro:strmsubregchar} it is strongly metrically
subregular iff $\alpha(\Lambda)>0$ and
$$
    \subreg{\Lambda}{\bar x}{\bar y}=\subreg{\Lambda}
    {\nullv}{\nullv}={1\over\alpha(\Lambda)}.
$$
Notice that, whenever $\X$ and $\Y$
are finite-dimensional spaces, $\alpha(\Lambda)>0$ holds iff
$\Kern\Lambda=\Lambda^{-1}(\nullv)=\{\nullv\}$, that is iff
$\Lambda$ is injective. This fact fails to remain true in abstract
normed space. Consider, for instance, the identity operator
${\rm Id}:\ell^1\longrightarrow\ell^\infty$ (the immersion of $\ell^1$
into $\ell^\infty$), which is injective,
and define for each $n\in\N$ the elements $x^n=\{x_k^n\}\in\ell^1$
as follows
\begin{eqnarray*}
     x_k^n=\left\{\begin{array}{ll}
                    1/n, & \hbox{ for } 1\le k\le n, \\
                     0, &  \hbox{ for } k\ge n+1,
                  \end{array} \right. \qquad n\in\N.
\end{eqnarray*}
It is clear that $\|x^n\|_{\ell^1}=1$, whereas $\|x^n\|_{\ell^\infty}
=1/n$ for every $n\in\N$. Consequently, one has
$$
    \alpha({\rm Id})=\inf_{\|u\|_{\ell^1}=1}\|{\rm Id}\, u\|_{\ell^\infty}\le
   \inf_{n\in\N}\|x^n\|_{\ell^\infty}=0,
$$
so ${\rm Id}$ is not strongly metrically subregular (anywhere).
On the other hand, it is clear that ${\rm Id}:\ell^1\longrightarrow
\ell^1$ is strongly metrically subregular, with $\subregmap{{\rm Id}}
{x}=1$, for every $x\in\ell^1$.

Recall that the
injectivity constant is connected with the Banach constant of linear
operators, through transposition. Namely, given $\Lambda(\X,\Y)$,
one defines
$$
    \beta(\Lambda)=\alpha(\Lambda^\top),
$$
where $\Lambda^\top$ stands for the transpose of $\Lambda$.
This allows one to link the notion of strong metric subregularity
with that of metric regularity, which is well known to amount to
openness (at a linear rate) in the case of linear operators.
More precisely, whenever $\Lambda$ is open, one has
$0<\beta(\Lambda)=\alpha(\Lambda^\top)$, so that $\Lambda^\top$
is strongly metrically subregular. For more details see \cite{Peno13}.
The current example helps also to illustrate the fact that,
whereas the appearance of sharp minimality for a given function
is a syntom of nonsmoothness, strong metric subregularity is a
property that may happen to take place for very nice (even linear)
mappings.
\end{example}

\begin{example}     \label{ex:smsubregex3}
Let $\varphi:\X\longrightarrow\R\cup\{+\infty\}$ be a proper, l.s.c.
convex function defined on a Banach space $\X$, whose dual is
indicated by $\X^*$. Let us denote by $\partial\varphi(\bar x)$ the
subdifferential of $\varphi$ at $\bar x\in\dom\varphi$
in the sense of convex analysis.
Generalizing a previous result valid in Hilbert spaces, in \cite{AraGeo14}
it is has been proved that the set-valued mapping $\partial\varphi:
\X\rightrightarrows\X^*$ is strongly metrically subregular at
$(\bar x,\bar x^*)\in\gr\partial\varphi$ if and only if there exist
positive $\gamma$ and $r$ such that
$$
    \varphi(x)\ge\varphi(\bar x)+\langle\bar x^*,x-\bar x\rangle+
    \gamma\|x-\bar x\|^2,\quad\forall x\in\ball{\bar x}{r},
$$
where $\langle\cdot,\cdot\rangle:\X^*\times\X\longrightarrow\R$
denotes the duality pairing $\X^*$ and $\X$.
In particular, in the case of a (global) minimizer $\bar x$ of
$\varphi$, $\partial\varphi$ is strongly metrically subregular at
$(\bar x,\nullv^*)$, where $\nullv^*$ stands for the null vector
of $\X^*$, iff
$$
    \varphi(x)\ge\varphi(\bar x)+\gamma\|x-\bar x\|^2,\quad\forall
    x\in\ball{\bar x}{r}.
$$
The last inequality formalizes a variational behaviour known 
as {\it quadratic growth condition}, which has been studied in
connection with second-order sufficient conditions in nonlinear
programming (see \cite{BonSha00}). Notice that, if a function admits
a sharp minimizer, it satisfies the quadratic growth condition
around that point, but the converse may not be true. 
Similar characterizations
of various metric regularity properties have been recently extended
to the Mordukhovich subdifferential mapping (see \cite{MorNgh13}).
Investigations by means of second-order variational analysis tools
revealed that they are also interrelated to the tilt-stability of local
minimizer (see \cite{MorNgh13,DrMoNg15}).
\end{example}

\begin{remark}
From Example \ref{ex:smsubregex3} it is possible to see at once
that if $\varphi$ is a proper, l.s.c. convex function, whose subdifferential
mapping is strongly metrically subregular at $(\bar x,\nullv^*)$,
where $\bar x\in\dom\varphi$ is a minimizer of $\varphi$, then
$\bar x$ turns out to be Tykhonov well-posed, namely every minimizing
sequence $\{x_n\}$ of $\varphi$ converges to $\bar x$. Now, it is
worth noting that the notion of strong metric subregularity
generalizes, yet in a local form, such a behaviour to solutions of
equations/inclusions. More precisely, if a set-valued mapping
$F:X\rightrightarrows Y$,
defining with $\bar y$ the inclusion $\bar y\in F(x)$, is strongly
metrically subregular at $(\bar x,\bar y)$, then for every sequence
$\{y_n\}$ in $Y$, with $y_n\longrightarrow\bar y$ as $n\to\infty$,
and for every sequence $\{x_n\}$ in $X$ of solutions of the
inclusions with data perturbed $y_n\in\ F(x)$, according to
(\ref{in:strmsubreg}) one finds
$$
   d(x_n,\bar x)\le\kappa\dist{\bar y}{F(x_n)}\le\kappa
    d(\bar y,y_n),
$$
so $x_n\longrightarrow\bar x$, provided that the elements
of $\{x_n\}$ fall in a proper neighbourhood of $\bar x$.
\end{remark}


\section{Perturbation stability}
\label{sec:3}

As it happens for other Lipschitzian properties of multifunctions,
a method for establishing criteria or conditions for the validity
of strong metric subregularity consists in analyzing its stability
in the presence of perturbations. Two results of this type
are presented in what follows, which are both proved through
the criterion discussed in Section \ref{sec:2}.

\begin{theorem}
Let $F:X\rightrightarrows Y$ be a set-valued mapping between
metric spaces and let $g:Z\longrightarrow X$ a mapping defined
on a metric space. Let $\bar z\in Z$ and let $(g(\bar z),\bar y)\in
\gr F$. Suppose that:

\noindent (i) $g$ is continuous at $\bar z$ and strongly metrically
subregular at $\bar z$;

\noindent (ii) $F$ is strongly metrically subregular at $(g(\bar z),
\bar y)$.

\noindent Then, their composition $F\circ g:Z\rightrightarrows
Y$ is strongly metrically subregular at $(\bar z,\bar y)$, and it
results in
$$
    \subreg{(F\circ g)}{\bar z}{\bar y}\le\subregmap{g}{\bar z}
    \cdot\subreg{F}{g(\bar z)}{\bar y}.
$$
\end{theorem}

\begin{proof}
Set $\bar x=g(\bar z)$. Since $F$ is strongly metrically subregular
at $(\bar x,\bar y)$, corresponding to an arbitrary $\kappa_F>
\subreg{F}{\bar x}{\bar y}$, there exists $r>0$ such that
\begin{eqnarray} \label{in:smsubregF}
   \dist{\bar y}{F(x)}\ge{1\over\kappa_F}\, d(x,\bar x),\quad\forall
   x\in\ball{\bar x}{r}.
\end{eqnarray}
Owing to hypothesis (i), corresponding to an arbitrary $\kappa_g>
\subregmap{g}{\bar x}$, there exists $\delta>0$ such that, up
to a further reduction of its value, if neeeded,
\begin{eqnarray}   \label{in:smsubregcong}
   d(\bar x,g(z))\ge{1\over \kappa_g}\, d(z,\bar z) \qquad\hbox{and}\qquad
    g(z)\in\ball{\bar x}{r},\quad\forall z\in\ball{\bar z}{\delta}.
\end{eqnarray}
From inequalities (\ref{in:smsubregcong}) and (\ref{in:smsubregF}),
one obtains
$$
   {\dist{\bar y}{F(g(z))}\over d(z,\bar z)}\ge {d(g(z),\bar x)
   \over \kappa_Fd(z,\bar z)}\ge {1\over \kappa_F\kappa_g},
   \quad\forall z\in\ball{\bar z}{\delta}\backslash\{\bar z\},
$$
whence $\cldis{F\circ g}(\bar z,\bar y)\ge (\kappa_F\kappa_g)^{-1}
>0$ follows. To complete the proof it remains to apply
Proposition \ref{pro:strmsubregchar}.
\hfill $\square$
\end{proof}

The following example shows that the continuity assumption
on the inner mapping can not be dropped out, in general.

\begin{example}      \label{ex:compcont}
Let $Z=X=Y=\R$ be endowed with the usual Euclidean metric structure.
Consider the functions $g:\R\longrightarrow\R$ and $F:\R\longrightarrow\R$,
defined respectively by
\begin{eqnarray*}
     g(z)=\left\{\begin{array}{ll}
                    0, & \hbox{ if } z=0, \\
                     2, &  \hbox{ otherwise}, 
                  \end{array} \right. \qquad\hbox{ and }\qquad
      F(x)=\left\{\begin{array}{ll}
                    |x|, & \hbox{ if } |x|\le 1, \\
                     2-|x|, &  \hbox{ otherwise}. 
                  \end{array} \right.
\end{eqnarray*}
Here $\bar z=\bar x=\bar y=0$.
It is evident that both $g$ and $F$ have a sharp minimizer
at $0$, whereas their composition $F\circ g\equiv 0$
does not.
\end{example}

Strong metric subregularity is not preserved under composition
of set-valued mappings, as shown by the next counterexample.

\begin{example}
In the same setting as in Example \ref{ex:compcont}, let $G:\R
\rightrightarrows\R$ be defined by
\begin{eqnarray*}
     G(z)=\left\{\begin{array}{ll}
                    \R, & \hbox{ if } z=0, \\
                     \{2\}, &  \hbox{ otherwise}, 
                  \end{array} \right. 
\end{eqnarray*}
and let $F$ be as in the previous example. As one readily checks,
$G$ is strongly metrically subregular at $(0,0)$. If composing $G$
and $F$, one finds
\begin{eqnarray*}
     (F\circ G)(z)=\left\{\begin{array}{ll}
                    (-\infty,1], & \quad\hbox{ if } z=0, \\
                     \{0\}, &  \quad\hbox{ otherwise}. 
                  \end{array} \right. 
\end{eqnarray*}
It is easily seen that $F\circ G$ fails to be strongly metrically
subregular at $(0,0)$. Notice that multifunction $G$ is upper
hemicontinuous at $0$.
\end{example}

For the next result, a slightly more structured setting is needed.

\begin{theorem}     \label{thm:smsubregstacalm}
Let $F:X\rightrightarrows Y$ be a set-valued mapping
defined on a metric space $X$ and taking values in a linear
metric space $Y$, whose metric is shift invariant.
If $F$ is strongly metrically subregular at $(\bar x,\bar y)\in\gr F$,
then for any mapping
$g:X\longrightarrow Y$ such that $\subreg{F}{\bar x}{\bar y}\cdot
\calmap{g}{\bar x}<1$, then the set-valued mapping $F+g$ is
strongly metrically subregular at $(\bar x,\bar y+g(\bar x))$
and it results in
\begin{eqnarray}   \label{in:estsubregpert}
  \subreg{(F+g)}{\bar x}{\bar y+g(\bar x)}\le{\subreg{F}{\bar x}{\bar y}
  \over 1-\subreg{F}{\bar x}{\bar y}\cdot\calmap{g}{\bar x}}.
\end{eqnarray}
\end{theorem}

\begin{proof}
Notice that, by virtue of the shift-invariance of the metric on $Y$,
for any $x\in X$ one has
$$
   \dist{\bar y}{F(x)}\le\dist{\bar y+g(\bar x)}{F(x)+g(x)}+
   d(g(x),g(\bar x)).
$$
Consequently, one obtains
\begin{eqnarray*} 
   \cldis{F+g}(\bar x,\bar y+g(\bar x))&\ge&\liminf_{x\to\bar x}
   {\dist{\bar y}{F(x)}-d(g(x),g(\bar x))\over d(x,\bar x)}  \\
   &\ge& \cldis{F}(\bar x,\bar y)-\limsup_{x\to\bar x}
   {d(g(x),g(\bar x))\over d(x,\bar x)} \\
   &\ge &{1\over\subreg{F}{\bar x}{\bar y}}-\calmap{g}{\bar x}>0.
\end{eqnarray*} 
The strong metric subregularity of $F+g$ at $(\bar x,\bar y
+g(\bar x))$ follows at once by the characterization
provided in Proposition \ref{pro:strmsubregchar},
whereas the estimate (\ref{in:estsubregpert}) is a
straightforward consequence of (\ref{eq:subregstdisrate}).
\hfill $\square$
\end{proof}

\begin{remark}
(i) The result provided by Theorem \ref{thm:smsubregstacalm}
on the persistence of strong metric subregularity under calm additive
perturbations can be found in \cite{DonRoc09} (see Theorem 3I.6),
formulated for multifunctions acting in finite-dimensional spaces,
with a different proof. It is worth noting that, since neither the
Ekeland variational principle nor the convergence of iteration
procedures are used in the proof of Theorem \ref{thm:smsubregstacalm},
metric completeness plays no role in the above robustness
phenomenon. Instead, it may be viewed as a direct consequence of
a stability behaviour for the local sharp minimality called
{\it superstability}, which was observed already by B.T. Polyak
(see \cite{Poly87}). Essentially, it means that a point preserves
its local minimality even in the presence of additive calm perturbations.
 
This makes the robustness of strong metric
subregularity different from the corresponding behaviour of
metric regularity, requiring on one hand metric completeness
and on the other hand the Lipschitz property of perturbations (see
\cite{DonRoc09,Mord06}).

(ii) Note that the shift-invariance assumption on the metric of $Y$
is not actually restrictive. Indeed, a result due to Kakutani ensures that
any linear metric space can be equivalently remetrized by a
shift-invariant metric (see Theorem 2.2.11 in \cite{Role87}).
\end{remark}


\section{Strong metric subregularity of nonsmooth mappings}
\label{sec:4}

The main subject of this paper is the strong metric subregularity
of (possibly) nonsmooth mappings $f:\X\longrightarrow\Y$.
To deal with them, throughout
this section $(\X,\|\cdot\|)$ and $(\Y,\|\cdot\|)$ are supposed to be
normed (vector) spaces. The (closed) unit ball and the unit sphere
in any normed space are indicated by $\Uball$ and $\Usfer$,
respectively, whereas, in the case of dual spaces, by $\Uball^*$ and
$\Usfer^*$, respectively.

\subsection{A criterium via first-order $\epsilon$-approximations}

Differentiability is a wise combination of linearity and approximation.
The approach of analysis considered in this subsection relies on
the employment of positively homogeneous (for short, p.h.) mappings
as an appealing substitute of derivatives (that are linear operators),
while calmness replaces the classical convergence of the
remainder term. To do so, set
$$
    \Hom(\X,\Y)=\{h:\X\longrightarrow\Y:\ \hbox{ p.h. and continuous at }
   \nullv\}.
$$

\begin{definition}    \label{def:epsapprox}
Let $f:\X\longrightarrow\Y$ be a mapping between normed spaces,
let $\bar x\in\X$ and let $\epsilon>0$. A mapping $h\in\Hom(\X,\Y)$
is said to be a {\it first-order $\epsilon$-approximation} of $f$ at $\bar x$ if
$$
    \calmap{(f-h(\cdot-\bar x))}{\bar x}<\epsilon.
$$
\end{definition}

\begin{remark}
Whenever $h$ is a first-order $\epsilon$-approximation of $f$
at $\bar x$, the mapping $f(\bar x)+h(\cdot-\bar x)$ is a special case
of what is called in \cite{DonRoc09} an ``estimator". Of course, first-order
$\epsilon$-approximation is a nonsmooth analysis notion, which
allows to include (Fr\'echet) differentiability. Indeed, note that if $f$
is Fr\'echet differentiable at $\bar x$, with derivative $\FDer
f(\bar x)\in\Lin(\X,\Y)$,
then $\FDer f(\bar x)$ is a first-order $\epsilon$-approximation
of $f$ at $\bar x$, for every $\epsilon>0$. P.h. functions and
mappings, or some special classes of them, have been utilized
as a rough material for constructing generalized derivatives
since the very birth of nonsmooth analysis (see, for instance,
\cite{Robi87,DemRub95}).
On the other hand, the idea of studying properties of nonlinear
mappings by means of ``approximate differentials", which avoid
differentiability assumptions, precedes even nonsmooth analysis
(see, for instance, \cite{HilGra27}).
\end{remark}

After having replaced linear operators with p.h. mappings,
the next step consists in extending to $\Hom(\X,\Y)$ the definition
of injective constant, by letting
$$
   \alphao(h)=\inf_{\|u\|=1}\|h(u)\|.
$$

\begin{remark}
By applying Proposition \ref{pro:strmsubregchar} it is readily seen
that $h\in\Hom(\X,\Y)$ is strongly metrically subregular at $\nullv$
if and only if it holds $\alphao(h)>0$. It is to be noted however
that, in contrast with the linear case, such a characterization
is not valid for the strong metric subregularity of $h$ at every point
of $\X$. Consider, for instance, the norm  function $\|\cdot\|:\X
\longrightarrow\R$, with the dimension of $\X$ being greater than $1$
(possibly infinite). Clearly, it is $\alphao(\|\cdot\|)=1$. Taking any element
$\bar u$ in the unit sphere $\Usfer$ of $\X$, one finds
\begin{eqnarray*}
    \grsl{\|\cdot\|}(\bar u)=\liminf_{x\to\bar u}{|\|x\|-\|\bar u\||
    \over \|x-\bar u\|}\le\sup_{\delta>0}\ \inf_{x\in\Usfer\cap\ball{\bar u}{\delta}
    \backslash\{\bar u\}}{|\|x\|-\|\bar u\||\over \|x-\bar u\|}=0,
\end{eqnarray*}
and hence $\|\cdot\|$ fails to be strongly metrically subregular
at $\bar u$, even though it is so at $\nullv$.
\end{remark}

\begin{theorem}     \label{them:epsapproxchar}
Let $f:\X\longrightarrow\Y$ be a mapping between normed
spaces and let $\bar x\in\X$. If $f$ is first-order $\epsilon$-approximated
at $\bar x$ by $h\in\Hom(\X,\Y)$, with $\alphao(h)>\epsilon$,
then $f$ is strongly metrically subregular at $\bar x$, and
$$
   \subregmap{f}{\bar x}\le {1\over \alphao(h)-\epsilon}.
$$
Vice versa, if $f$ is strongly metrically subregular at $\bar x$, then
for any mapping $h\in\Hom(\X,\Y)$ first-order $\epsilon$-approximating $f$ at
$\bar x$, with $\epsilon< \subregmap{f}{\bar x}$, it results in
$$
    \alphao(h)\ge {1\over \subregmap{f}{\bar x}-\epsilon},
$$
so $h$ is strongly metrically subregular at $\nullv$.
\end{theorem}

\begin{proof}
According to Definition \ref{def:epsapprox}, for both the
assertions of the thesis the respective hypotheses imply
$$
    \calmap{(f-h(\cdot-\bar x))}{\bar x}=\calmap{(h(\cdot-\bar x)-f)}
  {\bar x}<\epsilon.
$$
To prove the first assertion it suffices to apply Theorem
\ref{thm:smsubregstacalm}, with $F$ anf $g$ given by
$$
    F(x)=f(\bar x)+h(x-\bar x)\qquad\hbox{and}\qquad
     g(x)=f(x)-f(\bar x)-h(x-\bar x),
$$
respectively, and to observe that $F$ is strongly metrically
subregular at $\bar x$ iff $h$ is so at $\nullv$, while $\calmap{g}
{\bar x}=\calmap{(f-h(\cdot-\bar x))}{\bar x}$.

Analogoulsy, to prove the second assertion, it suffices to apply
once again Theorem \ref{thm:smsubregstacalm}, now with
$$
    F(v)=f(\bar x+v)+f(\bar x)\qquad\hbox{and}\qquad
     g(v)=h(v)-f(\bar x+v)-f(\bar x).
$$
Indeed, $f$ is strongly metrically subregular at $\bar x$ iff
$F$ is so at $\nullv$, whereas
$$
    \calmap{g}{\nullv}=\calmap{(h(\cdot-\bar x)-f)}{\bar x}.
$$
The quantitative estimates complementing the thesis are
direct consequences of inequality (\ref{in:estsubregpert}).
\hfill $\square$
\end{proof}

As a special case of Theorem \ref{them:epsapproxchar}
it is possible to derive the following criterion for smooth
mappings.

\begin{corollary}      \label{cor:smoothnsc}
A Fr\'echet differentiable mapping $f:\X\longrightarrow\Y$
between normed spaces is strongly metrically subregular
at $\bar x\in\X$ if and only if $\alpha(\FDer{f}(\bar x))>0$
and
$$
   \subregmap{f}{\bar x}\le {1\over \alpha(\FDer{f}(\bar x))}.
$$
In particular, if $\X$ and $\Y$ are finite-dimensional spaces,
$f$ is strongly metrically subregular at $\bar x$ if and only if
$$
    \Kern\FDer{f}(\bar x)=\{\nullv\}.
$$
\end{corollary}

\subsection{A sufficient condition via outer $\epsilon$-prederivative}

In order to introduce the next tool of nonsmooth analysis to be used,
recall that a set-valued mapping $F:\X\rightrightarrows\Y$ is said
to be p.h. if $\nullv\in F(\nullv)$ and $F(tx)=tF(x)$ for all $x\in\X$ and $t>0$.
In \cite{Ioff81} p.h. set-valued mappings have been used to define
a notion of generalized derivative. Below, a generalization of it, which
seems to be adequate for the purposes of the present analysis,
is introduced.

\begin{definition}     \label{def:preDermap}
Let $f:\X\longrightarrow\Y$ be a mapping between normed spaces,
let $\bar x\in\X$ and let $\Fan:\X\rightrightarrows\Y$ a p.h. homogeneous
set-valued mapping. Given $\epsilon>0$, $\Fan$ is said to be an
{\it outer} ({\it Fr\'echet}) {\it $\epsilon$-prederivative} of $f$ at $\bar x$
if there exists $\delta>0$ and a function $r:\delta\Uball\longrightarrow
[0,\epsilon]$ such that
\begin{eqnarray}     \label{def:oepreder}
    f(\bar x+v)-f(\bar x)\in\Fan(v)+r(v)\|v\|\Uball,\quad\forall v
    \in\delta\Uball.
\end{eqnarray}

\end{definition}

\begin{remark}     \label{rem:oepreder}
Recall that, according to \cite{Ioff81}, $\Fan$  is said to
be an outer prederivative of $f$ at $\bar x$ if (\ref{def:oepreder})
holds true with a function $r:\delta\Uball\longrightarrow[0,\epsilon]$
such that $\displaystyle\lim_{v\to\nullv}r(v)=0$.
In such an event, $\Fan$ is an outer $\epsilon$-prederivative
of $f$ at $\bar x$, for every $\epsilon>0$.
\end{remark}

Given a p.h. set-valued mapping $\Fan:\X\rightrightarrows\Y$,
to detect its strong metric subregularity at $(\nullv,\nullv)$,
it seems to be natural to introduce a injectivity constant notion
as follows
$$
   \alpha(\Fan)=\inf_{\|u\|=1}\dist{\nullv}{\Fan(u)}.
$$
In the light of Prosition \ref{pro:strmsubregchar}, it is readily seen
that $\Fan$  is strongly metrically subregular at $(\nullv,\nullv)$ if and
only $\alpha(\Fan)>0$ and, upon this condition, it is
$$
    \subreg{\Fan}{\nullv}{\nullv}={1\over\alpha(\Fan)}.
$$
By employing the above nonsmooth analysis tools, one can establish
the following sufficient condition for strong metric subregularity.

\begin{theorem}     \label{thm:Fansc}
Suppose that a mapping $f:\X\longrightarrow\Y$ between normed spaces
admits an outer $\epsilon$-prederivative $\Fan:\X\rightrightarrows\Y$
at $\bar x$, such that $\alpha(\Fan)>\epsilon$. Then, $f$ is strongly
metrically subregular at $\bar x$ and
$$
     \subregmap{f}{\bar x}\le {1\over \alpha(\Fan)-\epsilon}.
$$
If, in particular, $\Fan$ is an outer prederivative of $f$ at $\bar x$,
then then stricter estimate holds
$$
     \subregmap{f}{\bar x}\le {1\over \alpha(\Fan)}.
$$
\end{theorem}

\begin{proof}
Let positive $\epsilon$, $\delta$ and $r:\delta\Uball\longrightarrow
[0,\epsilon]$ be as in Definition \ref{def:preDermap}. Then one obtains
for every $x\in\ball{\bar x}{\delta}$
\begin{eqnarray}     \label{in:prederest}
    \|f(\bar x)-f(x)\| &=& d(\nullv,f(x)-f(\bar x))    \\
   &\ge &\dist{\nullv}{\Fan(x-\bar x)+r(x-\bar x)\|x-\bar x\|\Uball}, \nonumber
  \end{eqnarray}
On the other hand, observe that for any $v\in\delta\Uball\backslash
\{\nullv\}$ it is
\begin{eqnarray*}
   \dist{\nullv}{\Fan\left(v/\|v\|\right)+r\left(v\right)\Uball} &=& 
   \inf_{y\in\Fan(v/\|v\|),\, u\in\Uball}   \|y+r(v)u\|\\
   &\ge &\inf_{y\in\Fan(v/\|v\|),\, u\in\Uball} 
    \left[ \|y\|-|r(v)|\|u\|\right]  \\
    & \ge& \dist{\nullv}{\Fan(v/\|v\|)}-r(v).
\end{eqnarray*}
Therefore, letting $\bar y=f(\bar x)$, from inequality (\ref{in:prederest})
it follows
\begin{eqnarray*}
  \cldis{f}(\bar x,\bar y) &\ge&  \liminf_{x\to\bar x}
  {\dist{\nullv}{\Fan(x-\bar x)+r(x-\bar x)\|x-\bar x\|\Uball}
   \over\|x-\bar x\|}  \\
   &\ge&  \liminf_{x\to\bar x}
   \dist{\nullv}{\Fan\left({x-\bar x\over\|x-\bar x\|}\right)+r(x-\bar x)\Uball} \\
   &\ge & \inf_{\|u\|=1}\dist{\nullv}{\Fan(u)}-\limsup_{x\to\bar x}
   r(x-\bar x)= \inf_{\|u\|=1}\dist{\nullv}{\Fan(u)}-\epsilon.
\end{eqnarray*}
Thus, to show the first assertion in the thesis, it suffices to
apply the characterization stated in Proposition \ref{pro:strmsubregchar},
along with the estimate (\ref{eq:subregstdisrate}).
For the second assertion, on account of Remark \ref{rem:oepreder}
and of the last inequalities, one  has that
$$
     \cldis{f}(\bar x,\bar y) \ge\alpha(\Fan)-\epsilon,\quad
     \forall\epsilon\in (0,\alpha(\Fan)).
$$
which leads immediately to the estimate to be proved.
\hfill $\square$
\end{proof}

\begin{remark}
Among the p.h. set-valued mappings that can be used as prederivatives,
one can consider in particular those generated by a convex weakly closed
set of linear operators. In other terms, given a set ${\mathcal U}\subseteq
\Lin(\X,\Y)$ convex and closed with respect to the weak topology, let
$$
    \Fan(x)=\{y\in\Y:\ y=\Lambda x,\ \Lambda\in {\mathcal U}\}.
$$
According to \cite{Ioff81}, this is an example of fan. In this
case one has
$$
    \alpha(\Fan)=\inf_{\Lambda\in{\mathcal U}}\alpha(\Lambda).
$$
Notice that, whenever $f$ is Fr\'echet differentiable at $\bar x$, 
Definition \ref{def:preDermap} applies with $\Fan(v)=\{\FDer
f(\bar x)v\}$ and $r(v)=\|f(\bar x+v)-f(\bar x)-\FDer f(\bar x)v\|/\|v\|$,
being now ${\mathcal U}=\{\FDer f(\bar x)\}$.
Therefore the sufficient part of Corollary \ref{cor:smoothnsc}
can be achieved also from Theorem \ref{thm:Fansc}.
\end{remark}

\subsection{A scalarization approach}

Let $(\Y,\|\cdot\|)$ be now a normed space, which is partially
ordered by a relation $\le_\Y$ or, equivalently, by a convex 
cone $\Y_+\subseteq\Y$, in the sense that
$$
    y_1\le_\Y y_2\qquad\hbox{ iff }\qquad y_2-y_1\in\Y_+.
$$
In this setting, a mapping $f:\X\longrightarrow\Y$
is said to be $\Y_+$-convex if
$$
   f(tx_1+(1-t)x_2)\le_\Y tf(x_1)+(1-t)f(x_2),\quad\forall
   t\in [0,1],\quad\forall x_1,\, x_2\in\X.
$$
$\Y_+$-convex mappings are found quite easily in nature.
For instance, if $f:\X\longrightarrow\R^m$ is given by $f(x)=
(f_1(x),\dots,f_m(x))$, with each function $f_i:\X\longrightarrow\R$
being convex, then $f$ is $\R^m_+$-convex, where
$\R^m_+=\{y\in\R^m:\ y_i\ge 0,\ i=1,\dots, m\}$.
One immediately sees that if $f$ is $\Y_+$-convex and if
$y^*\in\Y^*_+=\{y^*\in\Y^*:\ \langle y^*,y\rangle\ge 0,\quad
\forall y\in\Y_+\}$, then each scalar function $y^*\circ f:\X
\longrightarrow\R$ is convex.

Since the scalarization approach to strong metric subregularity
exploits the connection of that property with the sharp minimality of scalarized
terms, the following characterization of a sharp minimizer
for a convex function will be useful in the sequel.

\begin{lemma}      \label{lem:wsminintsubd}
Let $\varphi:\X\longrightarrow\R\cup\{+\infty\}$ be a proper convex
function. An element $\bar x\in\dom\varphi$ is a (global) sharp
minimizer of $\varphi$ if and only if $\nullv^*\in\inte\partial
\varphi(\bar x)$. Moreover, it results in
\begin{eqnarray*}
   \grsl{\varphi}(\bar x)=\sup\{\rho>0:\ \rho\Uball^*\subseteq
   \partial\varphi(\bar x)\}.
\end{eqnarray*}
\end{lemma}

\begin{proof}
Necessity: suppose point $\bar x$ to be a sharp minimizer of $\varphi$
and take $\epsilon\in (0,\grsl{\varphi}(\bar x))$. Then, setting
$\rho=\grsl{\varphi}(\bar x)-\epsilon$ one gets
$$
   {\varphi(x)-\varphi(\bar x)\over\|x-\bar x\|}\ge\rho\ge
   \left\langle x^*,{x-\bar x\over\|x-\bar x\|}\right\rangle,
  \quad\forall x^*\in\rho\Uball^*,\ \forall x\in\X\backslash
   \{\bar x\}.
$$
This means that $\rho\Uball^*\subseteq\partial\varphi(\bar x)$
and, by arbitrariness of $\epsilon$, that $\grsl{\varphi}(\bar x)
\le\sup\{\rho>0:\ \rho\Uball^*\subseteq \partial\varphi(\bar x)\}$.

Sufficiency: from the definition of sugradient of $\varphi$
at $\bar x$, it is possible to deduce
$$
   {\varphi(x)-\varphi(\bar x)\over\|x-\bar x\|}\ge
   \sup_{x^*\in\partial\varphi(\bar x)}
    \left\langle x^*,{x-\bar x\over\|x-\bar x\|}\right\rangle,
   \quad\forall x\in\X\backslash
   \{\bar x\}.
$$
Since by hypothesis there exists $\rho>0$ such that $\rho\Uball^*
\subseteq\partial\varphi(\bar x)$, it is true that
$$
   \sup_{x^*\in\partial\varphi(\bar x)}\langle x^*,u\rangle\ge
   \sup_{x^*\in\rho\Uball^*}\langle x^*,u\rangle=\rho,\quad
   \forall u\in\X:\ \|u\|=1.
$$
From this and the previous inequality it is possible to conclude
that
$$
   \grsl{\varphi}(\bar x)\ge\rho>0.
$$
Actually, this shows that $\grsl{\varphi}(\bar x)\ge\sup\{\rho>0:\ \rho\Uball^*
\subseteq\partial\varphi(\bar x)\}$.
The proof is complete. \hfill $\square$
\end{proof}

To formulate the next condition for the strong metric subregularity
of a $\Y_+$-convex mapping $f:\X\longrightarrow\Y$ at $\bar x\in\X$,
define
$$
   \intrad{f}{\bar x}=\sup\{\rho>0:\ \rho\Uball^*
   \subseteq\partial(y^*\circ f)(\bar x),\ y^*\in\Usfer^*\cap\Y^*_+\}.
$$

\begin{theorem}       \label{thm:smsubconv}
Let $f:\X\longrightarrow\Y$ be a mapping between normed
spaces, with $\Y$ partially ordered by a cone $\Y_+$, and let
$\bar x\in\X$. Suppose that $f$ is $\Y_+$-convex and
$$
    \nullv^*\in\bigcup_{y^*\in\Usfer^*\cap\Y^*_+}\inte\partial
   (y^*\circ f)(\bar x).
$$
Then $f$ is strongly metrically subregular at $\bar x$. Moreover,
one has
$$
    \subregmap{f}{\bar x}\le{1\over \intrad{f}{\bar x}}.
$$
\end{theorem}

\begin{proof}
From the well-known dual representation of a norm
$$
    \|v\|=\sup_{y^*\in\Uball^*}\langle y^*,v\rangle=
    \sup_{y^*\in\Usfer^*}\langle y^*,v\rangle,
$$
one obtains
\begin{eqnarray*}
  {\|f(\bar x)-f(x)\|\over \|x-\bar x\|}&=&\sup_{y^*\in\Usfer^*}
  \left\langle y^*, {f(x)-f(\bar x)\over \|x-\bar x\|}\right\rangle  \\
   &\ge &\sup_{y^*\in\Usfer^*\cap\Y^*_+}
    {(y^*\circ f)(x)-(y^*\circ f)(\bar x)\over \|x-\bar x\|},\quad
    \forall x\in\X\backslash\{\bar x\}.
\end{eqnarray*}
Since by hypothesis there exist $\rho>0$ and $y^*_0\in\Usfer^*
\cap\Y^*_+$ such that $\rho\Uball^*\subseteq\partial(y^*_0\circ f)
(\bar x)$ and function $y^*_0\circ f$ is convex,
then in the light of Lemma \ref{lem:wsminintsubd},
$\bar x$ is a sharp minimizer of $y^*_0\circ f$. Consequently,
setting $\bar y=f(\bar x)$, from the last inequality one has
\begin{eqnarray*}
   \cldis{f}(\bar x,\bar y)\ge\grsl{(y^*_0\circ f)}(\bar x)\ge\rho.
\end{eqnarray*}
The proof of all assertions in the thesis is therefore completed by applying
Proposition \ref{pro:strmsubregchar}.
\hfill $\square$
\end{proof}

Theorem \ref{thm:smsubconv} demonstrates a typical use of the
scalarization method in the presence of convexity assumption.
It should be clear that this method extends its potential far
beyond convexity and can be employed in combination with
more general subdifferential constructions. For example, by
utilizing the Fr\'echet subdifferential, defined as
$$
    \Fsubd\varphi(\bar x)=\left\{x^*\in\X^*:\ \liminf_{x\to\bar x}
    {\varphi(x)-\varphi(\bar x)-\langle x^*,x-\bar x\rangle\over
     \|x-\bar x\|}\ge 0\right\}
$$
(for more details, the reader is referred to \cite{BorZhu05,Mord06,RocWet98,Schi07}),
the following finite-dimensional generalization of Lemma
\ref{lem:wsminintsubd} has been proved in \cite{Uder15}
(see Theorem 4 therein).

\begin{lemma}     \label{lem:wsminintFsub}
Let $\varphi:\R^n\longrightarrow\R\cup\{\pm\infty\}$ and $\bar x
\in\dom\varphi$. Then $\grsl{\varphi}(\bar x)>0$ if and only if
$\nullv^*\in\inte\Fsubd\varphi(\bar x)$.
\end{lemma}

The above lemma allows one to obtain the next result
valid for mappings defined in a finite-dimensional space.

\begin{theorem}
Given a mapping $f:\R^n\longrightarrow\Y$ and $\bar x\in\R^n$, if
$$
    \nullv^*\in\bigcup_{y^*\in\Usfer^*}\inte\Fsubd
   (y^*\circ f)(\bar x),
$$
then $f$ is strongly metrically subregular at $\bar x$.
\end{theorem}

\begin{proof}
The thesis can be achieved through the same argument as in the
proof of Theorem \ref{thm:smsubconv}. Indeed, by hypothesis there exists
$y^*_0\in\Usfer^*$ such that $\nullv^*\in\inte\Fsubd(y^*_0\circ f)(\bar x)$.
In the light of Lemma \ref{lem:wsminintFsub} this implies
\begin{eqnarray}    \label{in:scawsmin}
    \grsl{(y^*_0\circ f)}(\bar x)>0.
\end{eqnarray}
Since it is
\begin{eqnarray*}
  {\|f(\bar x)-f(x)\|\over \|x-\bar x\|}&=&\sup_{y^*\in\Usfer^*}
  \left\langle y^*, {f(x)-f(\bar x)\over \|x-\bar x\|}\right\rangle  \\
   &\ge &{(y^*_0\circ f)(x)-(y^*_0\circ f)(\bar x)\over \|x-\bar x\|},\quad
    \forall x\in\X\backslash\{\bar x\},
\end{eqnarray*}
by virtue of inequality (\ref{in:scawsmin}) it follows
\begin{eqnarray*}
   \cldis{f}(\bar x,\bar y)\ge\grsl{(y^*_0\circ f)}(\bar x)>0.
\end{eqnarray*}
Proposition \ref{pro:strmsubregchar} allows one to complete the
proof.
\hfill $\square$
\end{proof}


\section{An application}
\label{sec:5}

An application of the above exposed ideas and techniques
is going to be illustrated now, which concerns the stability
behaviour of solution mappings to generalized equations.
A generalized equation is a rather general problem that
is able to provide a proper framework for studying several specific
issues in mathematical analysis, having or not having a variational
nature. Among the others, let us mention optimality conditions
in constrained or unconstrained optimization, various types of
constraint systems, variational inequalities and complementarity
problems, equilibrium problems, differential inclusions.

Here, parameterized generalized equations are considered
that can be formalized as follows
$$
    \nullv\in f(p,x)+T(x),   \leqno\GEq
$$
where $f:P\times\X\longrightarrow\Y$ (sometimes referred to as
the base of $\GEq$) and $T:\X\rightrightarrows\Y$ (referred to as
the field) are the problem data. $(P,d)$ is a metric space where
the parameters vary, whereas $(\X,\|\cdot\|)$ and $(\Y,\|\cdot\|)$
are supposed to be normed vector spaces.
The solution mapping associated to $\GEq$ is the (generally)
set-valued mapping implicitly defined by
$$
    \Solv{f}{T}(p)=\{x\in\X:\ \nullv\in f(p,x)+T(x)\}.
$$
In this context, an issue of interest is how to certify and to
quantify a certain stability behaviour of $\Solv{f}{T}$, near a solution
$\bar x\in\Solv{f}{T}(\bar p)$. More precisely, here the stability behaviour
quantitatively described by the isolated calmness property is
investigated. This amounts to establish for solutions to $\GEq$
lying near a reference one a reaction, which is (directly) proportional
to the parameter variations.
Following the spirit of classical and more recent implicit
function theorems, this question is approached by
analyzing a semplified variant of $\GEq$, called {\it approximated
generalized equation} $\AGEq$, on which the main regularity
assumption is made. Of course, a $\AGEq$ can be defined
in several ways, depending on  the features of the problem
data. In what follows, dealing with a nonsmooth analysis
setting, the use of an adaptation of the outer $\epsilon$-prederivative
is proposed.

\begin{definition}   \label{def:parzeopreder}
Given $\epsilon>0$, a p.h. set-valued mapping $\Fan:\X\rightrightarrows
\Y$ is said to be a {\it partial outer $\epsilon$-prederivative}
of a mapping $f:P\times\X\longrightarrow\Y$ at $(\bar p,\bar x)$,
{\it uniformly with respect to $p$}, if there exist positive
$\delta$ and $\zeta$ and a function $r:P\times\delta\Uball\longrightarrow
[0,\epsilon]$ such that
$$
   f(p,x)\in f(p,\bar x)+\Fan(x-\bar x)+r(p,x-\bar x)\|x-\bar x\|\Uball,
   \quad\forall  x\in\ball{\bar x}{\delta},\ \forall p\in\ball{\bar p}{\zeta}.
$$
\end{definition}

Now, assuming that the base $f$ of $\GEq$ admits, for some $\epsilon>0$,
as a partial outer $\epsilon$-prederivative at $(\bar p,\bar x)$
a mapping $\Fan$, one can associate with $\GEq$ an
approximated generalized equation defined by
$$
    \nullv\in f(\bar p,\bar x)+\Fan(x-\bar x)+T(x).   \leqno\AGEq
$$
It turns out that the strong metric subregularity of the
mapping $\Fan+T$ is a key assumption to guarantee
the isolated calmness property of $\Solv{f}{T}$, as
below stated.

\begin{theorem}      \label{thm:impisocalm}
With reference to a generalized equation $\GEq$, let
$\bar x\in\Solv{f}{T}(\bar p)$. Suppose the data of $\GEq$
to satisfy the following assumptions:

\noindent (i) function $f(\cdot,\bar x)$ is calm at $\bar p$
with modulus $\calmap{f(\cdot,\bar x)}{\bar p}$;

\noindent (ii) $f$ admits a partial outer $\epsilon$-prederivative
$\Fan$ at $(\bar p,\bar x)$, uniformly with respect to $p$;

\noindent (iii) the set-valued mapping $x\rightrightarrows f(\bar p,
\bar x)+\Fan(x-\bar x)+T(x)$ is strongly metrically subregular at
$(\bar x,\nullv)$, with modulus $\subreg{(\Fan+T)}{\bar x}{\nullv}$,
such that
\begin{eqnarray}     \label{in:impisocalmcon}
     \epsilon\cdot\subreg{(\Fan+T)}{\bar x}{\nullv}<1.
\end{eqnarray}
Then, $\Solv{f}{T}$ has the isolated calmness property at $(\bar p,
\bar x)$ and the below estimate holds
\begin{eqnarray}
    \calm{\Solv{f}{T}}{\bar p}{\bar x}\le {\calmap{f(\cdot,\bar x)}{\bar p}
    \cdot\subreg{(\Fan+T)}{\bar x}{\nullv}\over 1-
    \epsilon\cdot\subreg{(\Fan+T)}{\bar x}{\nullv}}.
\end{eqnarray}
\end{theorem}

\begin{proof}
Take an arbitrary $\eta$ such that
$$
   0<\eta<{1\over\epsilon}-\subreg{(\Fan+T)}{\bar x}{\nullv},
$$
what is possible by virtue of condition (\ref{in:impisocalmcon}).
By the assumption (i), corresponding to $\eta$ there exists
$\zeta>0$ such that
$$
    \|f(p,\bar x)-f(\bar p,\bar x)\|\le (\calmap{f(\cdot,\bar x)}{\bar p}
    +\eta)d(p,\bar p),\quad\forall p\in\ball{\bar p}{\zeta}.
$$
By the assumption (ii), there exist $\tilde\zeta\in (0,\zeta)$, $\delta>0$
and  a function $r:P\times\delta\Uball\longrightarrow [0,\epsilon]$
such that
$$
     f(p,x)\in f(p,\bar x)+\Fan(x-\bar x)+r(p,x-\bar x)\|x-\bar x\|\Uball,
     \ \forall  x\in\ball{\bar x}{\delta},\ \forall p\in\ball{\bar p}{\tilde\zeta}.
$$
Consequently, one obtains
\begin{eqnarray*}
   &\dist{\nullv}{f(p,x)+T(x)}\ge \\
    &\dist{\nullv}{f(p,\bar x)+\Fan(x-\bar x)+
   r(p,x-\bar x)\|x-\bar x\|\Uball+T(x)}\ge   \\
     & \dist{\nullv}{f(p,\bar x)+\Fan(x-\bar x)+T(x)}-\epsilon\|x-\bar x\|\ge   \\
    &\dist{\nullv}{f(\bar p,\bar x)+(\calmap{f(\cdot,\bar x)}{\bar p}
    +\eta)d(p,\bar p)\Uball+\Fan(x-\bar x)+T(x)} \\
     &-\epsilon\|x-\bar x\|\ge   \\
    &\dist{\nullv}{f(\bar p,\bar x)+\Fan(x-\bar x)+T(x)}-(\calmap{f(\cdot,\bar x)}{\bar p}
    +\eta)d(p,\bar p)-\epsilon\|x-\bar x\|
\end{eqnarray*}
for every $x\in\ball{\bar x}{\delta}$ and $p\in\ball{\bar p}{\tilde\zeta}$,
wherefrom it follows
\begin{eqnarray*}
     \dist{\nullv}{f(\bar p,\bar x)+\Fan(x-\bar x)+T(x)} &\le& 
    \dist{\nullv}{f(p,x)+T(x)}   \\
   &+& (\calmap{f(\cdot,\bar x)}{\bar p}
    +\eta)d(p,\bar p)+\epsilon\|x-\bar x\|.
\end{eqnarray*}
Now, according to assumption (iii), since $\bar x$ is evidently
a solution to $\AGEq$, corresponding to $\eta>0$ there exists
$\tilde\delta\in (0,\delta)$ such that
\begin{eqnarray*}
    \|x-\bar x\| &\le& (\subreg{(\Fan+T)}{\bar x}{\nullv}+\eta)
   \dist{\nullv}{f(\bar p,\bar x)+\Fan(x-\bar x)+T(x)}     \\
   &\le & (\subreg{(\Fan+T)}{\bar x}{\nullv}+\eta)
     \biggl(\dist{\nullv}{f(p,x)+T(x)}  \\
    &+ & (\calmap{f(\cdot,\bar x)}{\bar p}
    +\eta)d(p,\bar p)+\epsilon\|x-\bar x\| \biggl)
\end{eqnarray*}
and hence
\begin{eqnarray*}
    (1-\epsilon (\subreg{(\Fan+T)}{\bar x}{\nullv}+\eta))\|x-\bar x\| &\le&
    (\subreg{(\Fan+T)}{\bar x}{\nullv}+\eta)   \\
    &\cdot &\biggl(\dist{\nullv}{f(p,x)+T(x)}     \\
    &+& (\calmap{f(\cdot,\bar x)}{\bar p}
    +\eta)d(p,\bar p) \biggl)
\end{eqnarray*}
for every $x\in\ball{\bar x}{\tilde\delta}$ and $p\in
\ball{\bar p}{\tilde\zeta}$.
As a consequence, whenever it is $x\in\Solv{f}{T}(p)\cap
\ball{\bar x}{\tilde\delta}$, it results in
$$
    \|x-\bar x\|\le {(\subreg{(\Fan+T)}{\bar x}{\nullv}+\eta)\cdot
      (\calmap{f(\cdot,\bar x)}{\bar p}+\eta)
 \over (1-\epsilon (\subreg{(\Fan+T)}{\bar x}{\nullv}+\eta))}
  \, d(p,\bar p),
$$
for every $p\in\ball{\bar p}{\tilde\zeta}$.
The last inequality shows that $\Solv{f}{T}$ has the isolated
calmness property at $(\bar p,\bar x)$ with 
$$
    \calm{\Solv{f}{T}}{\bar p}{\bar x}\le
    {(\subreg{(\Fan+T)}{\bar x}{\nullv}+\eta)\cdot
    (\calmap{f(\cdot,\bar x)}{\bar p}+\eta)\over
    (1-\epsilon (\subreg{(\Fan+T)}{\bar x}{\nullv}+\eta))}.
$$
From the last inequality and the arbitrariness of $\eta$, it is
possible to deduce the estimate in the thesis, thereby completing
the proof.
\hfill $\square$
\end{proof}

\begin{corollary}
With reference to a generalized equation $\GEq$, let
$\bar x\in\Solv{f}{T}(\bar p)$. Suppose the data of $\GEq$
to satisfy the following assumptions:

\noindent (i) function $f(\cdot,\bar x)$ is calm at $\bar p$
with modulus $\calmap{f(\cdot,\bar x)}{\bar p}$;

\noindent (ii') $f$ has a partial outer prederivative $\Fan$ at
$(\bar p,\bar x)$, uniformly with respect to $p$;

\noindent (iii) the set-valued mapping $x\rightrightarrows f(\bar p,
\bar x)+\Fan(x-\bar x)+T(x)$ is strongly metrically subregular at
$(\bar x,\nullv)$, with modulus $\subreg{(\Fan+T)}{\bar x}{\nullv}$.

\noindent Then $\Solv{f}{T}$ has the isolated calmness property
at $(\bar p,\bar x)$ and the stricter estimate 
\begin{eqnarray}
    \calm{\Solv{f}{T}}{\bar p}{\bar x}\le \calmap{f(\cdot,\bar x)}{\bar p}
    \cdot\subreg{(\Fan+T)}{\bar x}{\nullv}
\end{eqnarray}
holds.
\end{corollary}

\begin{proof}
The thesis can be easily achieved by applying Theorem 
\ref{thm:impisocalm}. Recall indeed that with assumption (ii') being
valid, $\Fan$ is an outer partial $\epsilon$-prederivative of $f$
at $(\bar p,\bar x)$, for any $\epsilon>0$. Then, it suffices to observe that, since
$\epsilon$ can be taken arbitrarily ``small", condition (\ref{in:impisocalmcon})
is fulfilled independently of the value of $\subreg{(\Fan+T)}{\bar x}{\nullv}$.
\hfill $\square$
\end{proof}

\begin{remark}
A result quite close to Theorem \ref{thm:impisocalm}, called ``implicit
mapping theorem with strong metric subregularity",  can be found in
\cite{DonRoc09} (Theorem 3I.12). Instead of prederivatives, partial
estimators (i.e. first-order $\epsilon$-approximations) are employed
there. In this regard, it is must be noted that the technique of proof
in Theorem \ref{thm:impisocalm} can be readily adapted to derive
a version of it, employing partial first-order $\epsilon$-approximations
of the base term.
\end{remark}

Theorem \ref{thm:impisocalm} reduces the study of the isolated
calmness property of $\Solv{f}{T}$ to the certification of the strong
metric subregularity of the set-valued mapping defining $\AGEq$.
The latter question is expected to be easier to be faced than
a direct study of $\Solv{f}{T}$, inasmuch as the former set-valued mapping
is explicitly defined in terms of problem data or their approximation,
while $\Solv{f}{T}$ can be hardly calculated in practice. Besides, in some
special case, the study of the strong metric subregularity of the
set-valued mapping defining $\AGEq$ may happen to be particularly
simple. Let us consider, as an example, the case in which the field
term $T$ happens to be single-valued near $\bar x$.

\begin{corollary}
Let $\bar x\in\Solv{f}{T}(\bar p)$ be a solution to $\GEq$.
Suppose that:

\noindent (i) function $f(\cdot,\bar x)$ is calm at $\bar p$
with modulus $\calmap{f(\cdot,\bar x)}{\bar p}$;

\noindent (ii) $f$ admits a partial outer $\epsilon$-prederivative
$\Fan$ at $(\bar p,\bar x)$, uniformly with respect to $p$;

\noindent (iii) $T$ is single-valued near $\bar x$ and calm
at $\bar x$, with modulus $\calmap{T}{\bar x}$;

\noindent (iv) the following condition holds
\begin{eqnarray}    \label{in:condsingval}
      \alpha(\Fan)-\calmap{T}{\bar x}>\epsilon.
\end{eqnarray}

\noindent Then $\Solv{f}{T}$ has the isolated calmness property
at $(\bar p,\bar x)$ and the following modulus estimate holds
\begin{eqnarray*}
    \calm{\Solv{f}{T}}{\bar p}{\bar x}\le {\calmap{f(\cdot,\bar x)}{\bar p}
    \over\alpha(\Fan)-\calmap{T}{\bar x}-\epsilon}.
\end{eqnarray*}
\end{corollary}

\begin{proof}
Observe first of all that $\Fan(\cdot-\bar x)$ is strongly metrically
subregular at $(\bar x,\nullv)$ iff $\Fan$ is so at $(\nullv,\nullv)$,
and  one has
$$
    \subreg{\Fan(\cdot-\bar x)}{\bar x}{\nullv}=
    \subreg{\Fan}{\nullv}{\nullv}.
$$
Under the current assumptions, it is possible to apply Theorem
\ref{thm:smsubregstacalm}, with $F=\Fan(\cdot-\bar x)$ and
$g=f(\bar p,\bar x)+T$. Indeed, it is clear that
$$
    \calmap{(f(\bar p,\bar x)+T)}{\bar x}=\calmap{T}{\bar x},
$$
so, in force of condition (\ref{in:condsingval}), it holds
$$
   \subreg{\Fan(\cdot-\bar x)}{\bar x}{\nullv}\cdot
   \calmap{(f(\bar p,\bar x)+T)}{\bar x}={\calmap{T}{\bar x}
    \over\alpha(\Fan)}<1.
$$
Consequently, the set-valued mapping $x\rightrightarrows f(\bar p,
\bar x)+\Fan(x-\bar x)+T(x)$ turns out to be strongly metrically
subregular at $(\bar x,\nullv)$, with
$$
    \subreg{(\Fan+T)}{\bar x}{\nullv}\le {1\over\alpha(\Fan)
  -\calmap{T}{\bar x}}.
$$
One is therefore in a position to apply Theorem \ref{thm:impisocalm},
as the validity of condition (\ref{in:impisocalmcon}) is ensured
by the assumption (\ref{in:condsingval}). Thus the proof is complete.
\hfill $\square$
\end{proof}

In the remaining part of this section, while continuing to
assume $T$ to be single-valued, a further result is presented,
which can be obtained via the scalarization approach.

\begin{theorem}
With reference to a generalized equation $\GEq$, let $\bar x\in
\Solv{f}{T}(\bar p)$. Suppose that:

\noindent (i) $T$ is single-valued near $\bar x$ and calm
at $\bar x$, with modulus $\calmap{T}{\bar x}$;

\noindent (ii) $f(\cdot,x)$ is calm at $\bar p$, uniformly with
respect to $x$ near $\bar x$, with modulus $\calmap{f(\cdot,x)}{\bar p}$;

\noindent (iii) the space $(\Y,\|\cdot\|)$ is partially ordered by
a cone $\Y_+$ and the mapping $f(\bar p,\cdot)$ is $\Y_+$-convex;

\noindent (iv) both the conditions
\begin{eqnarray}    \label{in:scaconbarp}
    \nullv^*\in\bigcup_{y^*\in\Usfer^*\cap\Y^*_+}\inte\partial
   (y^*\circ f)(\bar x)
\end{eqnarray}
and
\begin{eqnarray}     \label{in:subclm1}
    {\calmap{T}{\bar x}\over
     \intrad{f(\bar p,\cdot)}{\bar x}}<1
\end{eqnarray}
hold true.

\noindent Then $\Solv{f}{T}$ has the isolated calmness property
at $(\bar p,\bar x)$ and the following modulus estimate holds
\begin{eqnarray*}
    \calm{\Solv{f}{T}}{\bar p}{\bar x}\le
  {\calmap{f(\cdot,x)}{\bar p}
    \over \intrad{f(\bar p,\cdot)}{\bar x}-\calmap{T}{\bar x}}.
\end{eqnarray*}
\end{theorem}

\begin{proof}
Hypothesis (iii) and condition (\ref{in:scaconbarp}) allow one to
apply Theorem \ref{thm:smsubconv} to the mapping $f(\bar p,\cdot)$.
According to it, $f(\bar p,\cdot)$ turns out to be strongly metrically
subregular at $\bar x$ and it results in
$$
     \subregmap{f(\bar p,\cdot)}{\bar x}\le{1
     \over\intrad{f(\bar p,\cdot)}{\bar x}}.
$$
Now, since owing to condition (\ref{in:subclm1}) it is
$$
    \subregmap{f(\bar p,\cdot)}{\bar x}\cdot\calmap{T}{\bar x}<1,
$$
Theorem \ref{thm:smsubregstacalm} guarantees that the mapping $x
\mapsto f(\bar p,x)+T(x)$ is strongly metrically subregular at $(\bar x,\nullv)$
and that it results in
$$
    \subreg{(f(\bar p,\cdot)+T)}{\bar x}{\nullv}\le
   {1\over\intrad{f(\bar p,\cdot)}{\bar x}-\calmap{T}{\bar x}}.
$$
This means that, corresponding to $\eta>0$, there exists $r>0$
such that
$$
    \|x-\bar x\|\le {(1+\eta)\|f(\bar p,x)+T(x)\|\over
    \intrad{f(\bar p,\cdot)}{\bar x}-\calmap{T}{\bar x}},
    \quad\forall x\in\ball{\bar x}{r}.
$$
By taking account of hypothesis (i),
one has that for some $\zeta>0$  and $\tilde r\in (0,r)$ it holds
\begin{eqnarray*}
   \|x-\bar x\|&\le &{(1+\eta)\|f(\bar p,x)-f(p,x)\|+\|f(p,x)+T(x)\|\over
   \intrad{f(\bar p,\cdot)}{\bar x}-\calmap{T}{\bar x}}    \\
   &\le &{(1+\eta)[(\calmap{f(\cdot,x)}{\bar p}+\eta)d(p,\bar p)+\|f(p,x)+T(x)\|]\over
    \intrad{f(\bar p,\cdot)}{\bar x}-\calmap{T}{\bar x}} 
\end{eqnarray*}
for every $x\in\ball{\bar x}{\tilde r}$ and $p\in\ball{\bar p}{\zeta}$.
Thus, if taking $x\in\ball{\bar x}{\tilde r}\cap\Solv{f}{T}(p)$, one finds
$$
    \|x-\bar x\|\le{(1+\eta)(\calmap{f(\cdot,x)}{\bar p}+\eta)\over
   \intrad{f(\bar p,\cdot)}{\bar x}-\calmap{T}{\bar x}}\,d(p,\bar p),
   \quad\forall p\in\ball{\bar p}{\zeta}.
$$
The last inequality shows that $\Solv{f}{T}$ fulfils the isolated
calmness property at $(\bar p,\bar x)$ and, by arbitrariness
of $\eta$, it allows one to achieve the asserted modulus estimation.
\hfill $\square$
\end{proof}


\section{Conclusions}
\label{sec:6}

The appraoch of analysis proposed in this paper shows that several techniques
for detecting strong metric subregularity of nonsmooth mappings can be
derived from a unique elementary criterion, based on the notion of
steepest displacement rate, which can be formulated already in a metric
space setting. This criterion, besides providing a unifying scheme of analysis
with transparent proofs, emphasizes the variational nature of the
property under study. Optimization (especially, nondifferentiable
optimization) is well recognized as a field where many results and constructions
of set-valued analysis are fruitfully applied.
The findings of the present study should contribue to the make it
evident that, simmetrically, nondifferentiable optimization can provide
useful insights and methods for investigating properties of multifunctions,
some of them not necessarily related to extremum problems.
This seems to agree with the very spirit of the Euler's variational faith.




\end{document}